\documentclass[12pt,a4paper]{article}
\usepackage{amsmath}
\usepackage{amssymb}
\usepackage{amsfonts}
\usepackage{amsthm}
\usepackage{relsize}
\usepackage{setspace}
\usepackage[left=3cm, right=3cm]{geometry}
\usepackage{enumerate}
\usepackage{url}
\usepackage{xspace}
\usepackage{mathrsfs}
\usepackage{tocloft}

\usepackage{hyperref}
\usepackage{authblk}
\usepackage[usenames, dvipsnames]{color}

\newtheorem{thm}{Theorem}[section]
\newtheorem{lem}[thm]{Lemma}
\newtheorem{prop}[thm]{Proposition}
\newtheorem{cor}[thm]{Corollary}

\theoremstyle{definition}

\newtheorem*{rem*}{Remark}

\newcommand{\defbold}{\textbf}

\newcommand{\inv}{^{-1}}

\newcommand{\CC}{\mathrm{C}}

\newcommand{\N}{\mathrm{N}}

\newcommand{\Z}{\mathrm{Z}}
\newcommand{\ZZ}{\mathbf{Z}}

\newcommand{\tdlc}{t.d.l.c.\@\xspace}

\newcommand{\con}{\mathrm{con}}

\newcommand{\nub}{\mathrm{nub}}

\newcommand{\prb}{\mathrm{par}}
\newcommand{\bd}{\partial}

\newcommand{\Aut}{\mathrm{Aut}}

\newcommand{\bN}{\mathbf{N}}

\newcommand{\bZ}{\mathbf{Z}}

\newcommand{\mcB}{\mathcal{B}}

\begin{document}

\title{Limits of contraction groups and the Tits core}

\author[1]{Pierre-Emmanuel Caprace\thanks{F.R.S.-FNRS research associate, supported in part by the ERC (grant \#278469)}}
\author[2]{Colin D. Reid\thanks{Supported in part by ARC Discovery Project DP120100996}}
\author[2]{George A. Willis\thanks{Supported in part by ARC Discovery Project DP0984342}}

\affil[1]{Universit\'e catholique de Louvain, IRMP, Chemin du Cyclotron 2, bte L7.01.02, 1348 Louvain-la-Neuve, Belgique}
\affil[2]{University of Newcastle, School of Mathematical and Physical Sciences, Callaghan, NSW 2308, Australia}

\date{April 18, 2013}
\maketitle



\begin{abstract}
The \textbf{Tits core} $G^\dagger$ of a totally disconnected locally compact group $G$ is defined as the abstract subgroup generated by the closures of the contraction groups of all its elements. We show that a dense subgroup is normalised by the Tits core if and only if it contains it. It follows that every dense subnormal subgroup contains the Tits core. In particular, if $G$ is topologically simple, then the Tits core is abstractly simple, and when $G^\dagger$ is non-trivial, it is the unique minimal dense normal subgroup. The proofs are based on the fact, of independent interest, that the map which associates  to an element the closure of its contraction group is continuous.  
\end{abstract}

\section{Introduction}

In his seminal paper \cite{Tits64}, J.~Tits uses his notion of BN-pairs to show the following result: \emph{if $\mathbf G$ is a simple algebraic group over a field $k$, then every subgroup of $\mathbf G(k)$ normalised by $\mathbf G(k)^\dagger$ either is central, or contains $\mathbf G(k)^\dagger$}. The group  $\mathbf G(k)^\dagger$ is defined to be the subgroup of $\mathbf G(k)$ generated by the unipotent radicals of the $k$-defined parabolic subgroups. Tits deduces in particular that the quotient of $\mathbf G(k)^\dagger$ by its centre is simple as an abstract group.  

In the present paper, we associate a subgroup  $G^\dagger$ to an arbitrary locally compact group $G$ and call it the \textbf{Tits core} of $G$. It is defined as the subgroup of $G$ generated by the closures of the contraction groups of all elements of $G$. Recall that the  \defbold{contraction group} of an element $g \in G$ is given by
\[ \con(g) = \{ u \in G \mid g^nug^{-n} \rightarrow 1 \text{ as } n \rightarrow +\infty\}.\]
Thus by definition we have
$$
G^\dagger = \langle \overline{\con(g)} \; | \; g \in G\rangle.
$$
The Tits core is clearly normal, and even characteristic, in $G$, but it need not be closed \emph{a priori}. 

In case $G$ is the group of $k$-rational points of a simple algebraic group $\mathbf G$ over a local field $k$, the contraction group of each element coincides with the unipotent radical of some $k$-defined parabolic subgroup, and is in particular closed (see \cite[Lemma~2.4]{Prasad}). It follows that $G^\dagger=  \mathbf G(k)^\dagger$ in this case; this justifies our choice of terminology for $G^\dagger$. 

The main results of this paper may be viewed as analogues of the aforementioned theorem by Tits in the case of totally disconnected locally compact groups. For the sake of brevity, we shall write   \emph{\tdlc} for \emph{totally disconnected locally compact}. 

\begin{thm}\label{thm:condensenorm}
Let $G$ be a \tdlc group and let $D$ be a dense subgroup of $G$. If $G^\dagger$ normalises $D$, then $G^\dagger \le D$. 
\end{thm}

The following consequence is immediate by induction on the length of a subnormal chain: 

\begin{cor}\label{cor:Titscore}
Let $G$ be a \tdlc  group. Then every dense subnormal subgroup of $G$ contains the Tits core  $G^\dagger$. 
\end{cor}

A slightly weaker form of the latter result has recently been used by T.~Marquis \cite{Marquis} to prove that irreducible complete Kac--Moody groups over finite fields are abstractly simple, thereby settling the main remaining open case of a question of J. Tits going back to \cite{Tits89}. 

Recall that a topological group is called \textbf{topologically simple} if its only closed normal subgroups are the identity subgroup and the whole group. The following is immediate from Corollary~\ref{cor:Titscore}.

\begin{cor}\label{cor:Titscore=dnc}
Let $G$ be a topologically simple   \tdlc group. If  the Tits core $G^\dagger$ is non-trivial, then it is the unique minimal dense subnormal subgroup of $G$. 
\end{cor}

As in the aforementioned situation studied by Tits, we obtain a result of abstract simplicity, which is immediate from Corollary~\ref{cor:Titscore=dnc}.

\begin{cor}\label{cor:Titscore:abssimple}
Let $G$ be a  \tdlc   group. If $G$ is topologically simple, then  the Tits core  $G^\dagger$ is abstractly simple. 
\end{cor}

A \tdlc group $G$ with trivial Tits core will be called \textbf{anisotropic}; equivalently $G$ is anisotropic if all its elements have trivial contraction group. This is certainly the case if $G$ is discrete, or compact, or   if every element of $G$ is contained in a compact subgroup. It should be emphasised that non-discrete topologically simple groups may also be anisotropic: indeed, the examples of topologically simple  \tdlc groups constructed in \cite[\S3]{Willis} are all unions of ascending chains of compact open subgroups, hence anisotropic. Notice however that such groups are not compactly generated; it is in fact an interesting open problem to determine whether  a non-compact compactly generated \tdlc group without infinite discrete quotient can be anisotropic. 

In any case, the Tits core is contained in the kernel of all anisotropic quotients of $G$. More precisely, results by Baumgartner--Willis \cite{BaumgartnerWillis} imply the following characterisation. 

\begin{prop}\label{prop:AnisotropicResidual}
Let $G$ be a \tdlc group and $N$ be a closed normal subgroup. Then $G/N$ is anisotropic if and only if $G^\dagger \leq N$. In particular $\overline{G^\dagger}$ is the unique smallest closed normal subgroup of $G$ affording an anisotropic quotient. 
\end{prop}

Although the results above highlight an analogy between the Tits core of a \tdlc group  and that of a simple algebraic group, the methods of proofs are necessarily completely different (in particular a \tdlc group need not have a BN-pair, even if it is topologically simple). 
The main tool in proving Theorem~\ref{thm:condensenorm} is  an analysis of the behaviour of the contraction groups associated to a converging sequence of elements in $G$. In order to describe our results in this direction, we recall that the   \textbf{nub} of the element $g$ in  a \tdlc group $G$ is defined by
$$
\nub(g) = \overline{\con(g)} \cap \overline{\con(g\inv)}.
$$ 
The nub is a compact subgroup normalised by $g$; it coincides with the unique maximal compact subgroup of $G$ normalised by $g$ and on which $g$ acts ergodically (see \cite{WilNub} and Theorem~\ref{wilnub} below). A key property of the nub    is that $\nub(g)$ is trivial if and only if $\con(g)$ is closed (see \cite[Theorem~3.32]{BaumgartnerWillis} and \cite{Jaworski}).

\begin{thm}\label{thm:connub:converge}
Let $G$ be a \tdlc group and $(g_n) $ be a net of elements of $G$ converging to some limit $g \in G$.  Then for $n$ sufficiently large, there exist elements $t_n$ and $r_n$ of $G$ such that $t_n,r_n \to 1$ as $n \to \infty$ and such that
\[ \con(g_n) = t_n\con(g)t\inv_n; \; \nub(g_n) = r_n\nub(g)r\inv_n.\]
\end{thm}

Denoting by $\mathcal S(G)$ the compact space of closed subgroups of $G$ endowed with the Chabauty topology, we obtain the following consequence.

\begin{cor}\label{cor:Chabauty}
Let $G$ be a \tdlc group. Then the maps  
$$
G \to \mathcal S(G) : g \mapsto \overline{\con(g)} \hspace{.5cm} \text{and} \hspace{.5cm} G \to \mathcal S(G) :  g \mapsto \nub(g)
$$ 
are both continuous.
\end{cor}

\section{Preliminaries}

Relatively little is known about the overall topological dynamics of a \tdlc group acting on itself.  By contrast, there is a well-developed theory of the dynamics of $\ZZ$ acting on a \tdlc group by continuous automorphisms.  We recall some basic concepts and results in this area which will be needed in the sequel; the main references are  \cite{BaumgartnerWillis}, \cite{Willis94} and \cite{WilNub}. Notice that the results from   \cite{BaumgartnerWillis} are stated under the hypothesis that the ambient groups are metrisable; this hypothesis has been removed by Jaworski~\cite{Jaworski}; we shall therefore freely refer to the results of \cite{BaumgartnerWillis} without any further comment about metrisability.

 For the rest of this section, fix a \tdlc group $G$.  We only allow automorphisms that preserve the topology of $G$.

Given an automorphism $f$ of $G$, the \defbold{contraction group} $\con(f)$ of $f$ on $G$ is given by
\[ \con(f) = \{ u \in G \mid f^n(u) \rightarrow 1 \text{ as } n \rightarrow +\infty\}.\]

One sees that $\con(f)$ is a subgroup of $G$, although not necessarily a closed subgroup.

Given a subgroup $U$ of $G$, define the following subgroups:
\[ U_- := \bigcap_{i \le 0}f^i(U); \quad U_+ := \bigcap_{i \ge 0}f^i(U); \quad U_0 := \bigcap_{i \in \bZ}f^i(U);\]
\[ U_{--} := \bigcup_{j \le 0} f^j(U_-); \quad U_{++} := \bigcup_{j \ge 0} f^j(U_+).\]

Say $U$ is \defbold{tidy above for $f$} if $U = U_+U_-$.  Say $U$ is \defbold{tidy below for $f$} if $U_{--}$ and $U_{++}$ are closed subgroups of $G$.  A \defbold{tidy subgroup for $f$} is an open compact subgroup of $G$ that is tidy both above and below for $f$.

\begin{thm}\label{tidythm}
Let $U$ be an open compact subgroup of $G$.
\begin{enumerate}[(i)]
\item There is some integer $k$ such that $\bigcap^k_{i=0} f^i(U)$ is tidy above for $f$.  Consequently $U_+U_-$ contains an open subgroup of $U$.

\item There is an open compact subgroup of $G$ that is tidy for $f$.

\item We have $U_{--} = \con(f)U_0$ (and similarly $U_{++} = \con(f\inv)U_0$).  In particular, $\con(f) = \con(f\inv) = 1$ if and only if the set of open compact subgroups of $G$ that are $f$-invariant forms a base of neighbourhoods of the identity.

\item We have  $U_{-} = (\con(f) \cap U_-)U_0$ and $U_+ = (\con(f\inv) \cap U_+) U_0$. 

\item If in addition $U$ is tidy for $f$, then we have $U_{-} = (\con(f) \cap U)U_0$ and $U_+ = (\con(f\inv) \cap U) U_0$. 
\end{enumerate}
\end{thm}

\begin{proof}
For (i) and (ii), see \cite{Willis94} and \cite{Willis01}. The first assertion of part (iii) is \cite[Proposition~3.16]{BaumgartnerWillis}. The second follows from the first, combined with (i). Part (iv) follows from (iii) since $U_0 \leq U_- \leq U_{--}$, and similarly for $U_+$. Finally, for (v) we observe that by \cite[Lemma~3]{Willis94} the group $U_{--}$ is closed if and only if $U_{--} \cap U = U_-$. Thus if $U$ is tidy for $f$, then (iii) yields $U_{-} = U_{--} \cap U =  (\con(f) \cap U)U_0$ since $U_0 \leq U$. The proof for $U_+$ is similar. 
\end{proof}

A more recent development in the dynamical theory of automorphisms of \tdlc groups is the \textbf{nub} of an automorphism.  As usual, we will define the nub of a group element to be the nub of the corresponding inner automorphism.  Before presenting several characterisations of the nub, we introduce one more subgroup associated with an automorphism $f$ of a \tdlc group $G$, namely  the \defbold{parabolic group} $\prb(f)$. It is defined to consist of those elements $x \in G$ such that the set ${\{f^n(x) \mid n \in \bN\}}$ is relatively compact.  By \cite[Proposition~3]{Willis94}, the parabolic subgroup $\prb(f)$ is always  closed.

\begin{thm}\label{wilnub}
Let $G$ be a \tdlc group and let $f$ be an automorphism of $G$.  Then there is an $f$-invariant compact subgroup $\nub(f)$ of $G$, which is equal to each of the following sets:
\begin{enumerate}[(i)]
\item the intersection of all tidy subgroups for $f$;
\item the closure of the set $\mathrm{bco}(f) := \con(f) \cap \prb(f\inv)$;

{
\item the intersection $\overline{ \con(f)} \cap \prb(f\inv)$;

\item the intersection $\overline{\con(f)} \cap \overline{\con(f\inv)}$;
}

\item the intersection $\bigcap_{V \in \mcB(G)} \overline{\mathrm{rbco}(f,V)}$, where
\[ \mathrm{rbco}(f,V) := \{ x \in \prb(f\inv) \mid \exists N \; \forall n \ge N \; f^n(x) \in V \} ;\]
\item the largest compact, $f$-stable subgroup of $G$ having no relatively open $f$-stable subgroups;
\item the largest compact, $f$-stable subgroup of $G$ on which $f$ acts ergodically.
\end{enumerate}
\end{thm}

{
\begin{proof}
Assertions (i), (ii), (v), (vi),  (vii) are contained in \cite{WilNub} Theorem 4.12. For assertion (iii) and (iv), see  respectively Lemma~3.29 and Corollary~3.27 in \cite{BaumgartnerWillis}.
\end{proof}

\begin{proof}[Proof of Proposition~\ref{prop:AnisotropicResidual}]
It is clear that the image of $G^\dagger$ in $G/N$ is contained in $(G/N)^\dagger$. Thus if $G/N$ is anisotropic, then $G^\dagger \leq N$. 

For each $g \in G$, we denote by $\con(g / N)$  the inverse image in $G$ of the contraction group of $gN$ in $G/N$. By \cite[Theorem~3.8]{BaumgartnerWillis}, we have $\con(g / N) = \con(g)N$.  Therefore,  if $\con(g/N)/N$ is non-trivial, then $\con(g) $ is not contained in $N$. In particular, if $G/N$ is not anisotropic, then $G^\dagger$ is not contained in $N$.
\end{proof}

\section{Limits of contraction groups}

The goal in this section is prove Theorem~\ref{thm:connub:converge} from the introduction. This will be achieved after a series of technical preparations. 

We start with a  lemma which  is a more precise version of \cite[Lemma 14]{Willis94}. The lemma is used there in the proof that, if $U$ is tidy for $g$ then it is also tidy for every element in $gU$ and that the scale is constant on $gU$. 

\begin{lem}\label{lem:conjugator}
Let $g\in G$ and let $U$ be an open compact subgroup of $G$ that is tidy above for $g$. Then for every $u\in U$ there is $t\in U_+ \cap \con(g\inv)$ such that, for every $k\geq0$,
$$
t^{-1}(gu)^ktg^{-k} \in U.
$$
\end{lem}

\begin{proof}
It will be shown by induction that for every $n\geq0$ there is $t_n\in U_+$ with 
\begin{equation}
\label{eq:conjugator}
t_n^{-1}(gu)^kt_n = b_{n,k}g^k,\text{ with }b_{n,k}\in U,\text{ for every }k\in\{0,1,\dots, n\}.
\end{equation}

The base case of the induction, when $n=0$, is certainly true. Suppose that (\ref{eq:conjugator}) has been established for $n$. The element $t_{n+1}$ will be constructed as $t_{n+1} = t_ny$, where $y\in g^{-n}U_+g^n$. 

Notice first that, for any choice of $y \in g^{-n}U_+g^n$ and all $k\in \{0,1,\dots, n\}$, we have
$$
t_{n+1}^{-1}(gu)^kt_{n+1} = y^{-1}b_{n,k}g^ky = b_{n+1,k}g^k
$$
where $b_{n+1,k} = y^{-1}b_{n,k}g^kyg^{-k}$ belongs to $U$. Now, to find suitable $y$, consider
$$
t_{n+1}^{-1}(gu)^{n+1}t_{n+1} =  y^{-1}t_n^{-1}(gu)(gu)^nt_ny = y^{-1}t_n^{-1}g(ut_nb_{n,n})g^ny.
$$
Since $ut_nb_{n,n}$ belongs to $U$ which is tidy above for $g$, we have $ut_nb_{n,n} = w_-w_+$ with $w_{\pm}\in U_{\pm}$. Put $y = g^{-n}w_+^{-1}g^n$. Then
$$
 t_{n+1}^{-1}(gu)^{n+1}t_{n+1} = t_{n+1}^{-1}gw_-g^n = b_{n+1,n+1}g^{n+1}
 $$
 with $b_{n+1,n+1} = y^{-1}t_n^{-1}gw_-g^{-1}$ in $U$ as required.
 
 Since $U_+$ is compact, the sequence $(t_n)_{n \in \bN}$ has a subsequence converging to an element $s \in U_+$.  We see that
 $$
s^{-1}(gu)^ksg^{-k} \in U \quad \forall k \ge 0
$$
for this choice of $s$.  Moreover, note that we can freely replace $s$ with $sv$ given some $v \in U_0$, since $U_0$ is normalised by $\langle g \rangle$ and hence  $[g^l,v] \in U$ for all $l \in \bZ$.  Now by Theorem~\ref{tidythm}(iv), we have
\[ U_+ = (U_+ \cap \con(g\inv))U_0,\]
hence by a suitable choice of $v$, we obtain $t = sv \in U_+ \cap \con(g\inv)$ with the desired property.
\end{proof}

\begin{cor}\label{cor:conjugate_contraction}
Let $g$, $u$ and $t$ be as in Lemma~\ref{lem:conjugator}. Then $t\con(g)t^{-1} = \con(gu)$.  
\end{cor}
\begin{proof}
For $k\geq0$, there is $b_k \in U$ such that $(gu)^k t = t b_k g^k$. Therefore, for each $c\in \con(g)$, we have
$$
(gu)^k(tct^{-1})(gu)^{-k} = tb_k ( g^k c g^{-k} ) (tb_k)^{-1}.
$$ 
 The right side converges to the identity as $k\to\infty$ because $tb_k$ belongs to $U$, which has a  neighbourhood base of the identity comprising  normal subgroups. Hence $t\con(g)t\inv \leq \con(gu)$.  Similarly, given $c \in \con(gu)$, then, for $k\geq0$, $$
g^k (t\inv ct) g^{-k} = b\inv_kt\inv((gu)^k c (gu)^{-k} )tb_k
$$ 
showing that $t\inv c t \in \con(g)$, so $\con(g) \ge t\inv\con(gu)t$ and hence $t\con(g)t^{-1} = \con(gu)$.
\end{proof}

\begin{lem}\label{lem:conjugator2}
Let $g\in G$ and let $U$ be an open compact subgroup of $G$ that is tidy above for $g$.  Then for every $u\in U\cap g^{-1}Ug$ there is $r\in U$ such that, for every $k\in\mathbf{Z}$,
\begin{equation}
\label{eq:conjugator2}
r^{-1}(gu)^kr = b_kg^k\text{ with }b_k\in U.
\end{equation}
\end{lem}
\begin{proof}
Applying Lemma~\ref{lem:conjugator} yields $t\in U_+$ such that  (\ref{eq:conjugator2}) holds for $k\geq0$ with $r=t$. Applying the lemma a second time replacing $g$ by $g^{-1}$ and $u$ by $gu^{-1}g^{-1}$, which belongs to $U$, yields $s\in U_-$ such that (\ref{eq:conjugator2}) holds for $k\leq 0$ with $r=s$. 

Since $s^{-1}t$ belongs to $U$, it is equal to $v_+v_-$ with $v_\pm\in U_\pm$. Put $r = tv_-^{-1} = sv_+$. Then for $k\geq0$
$$
r^{-1}(gu)^kr = v_-t^{-1}(gu)^ktv_-^{-1} = v_-b_kg^k v_-^{-1} = b_k'g^k
$$
with $b'_k = v_-b_kg^k v_-^{-1}g^{-k}$, which belongs to $U$, and for $k\leq0$
$$
r^{-1}(gu)^kr = v_+^{-1}s^{-1}(gu)^ksv_+ = v_+^{-1}b_kg^k v_+ = b_k'g^k
$$
with $b_k' = v_+^{-1}b_kg^k v_+g^{-k}$ which belongs to $U$. 
\end{proof}

\begin{cor}
\label{cor:nub_conjugate}
Let $g$, $u$ and $r$ be as in Lemma~\ref{lem:conjugator2}. Then $r\nub(g)r^{-1} = \nub(gu)$.  

If $u_n\to 1_G$, then $r_n$ with $r_n\nub(g)r_n^{-1} = \nub(gu_n)$ may be chosen for each $n$ such that $r_n\to 1_G$ as $n\to\infty$.
\end{cor}

\begin{proof}
By Theorem~\ref{wilnub} (ii) we have $\nub(g) = \overline{\mathrm{bco}(g)}$.  It follows from Lemma~\ref{lem:conjugator2} that $r\mathrm{bco}(g)r^{-1} = \mathrm{bco}(gu)$, whence the claim.
\end{proof}

\begin{proof}[Proof of Theorem~\ref{thm:connub:converge}]
Set $u_n := g\inv g_n$, so that $u_n \to 1$ as $n \to \infty$.  The tidying procedure (Theorem~\ref{tidythm}(i)) ensures that every open neighbourhood   of the identity contains a compact open subgroup that is tidy above for $g$. In particular we may find a (possibly uncountable) chain $(U_n)$ of open compact subgroups, all tidy above for $g$, with $\bigcap^\infty_n U_n = 1$.   Upon replacing $(g_n)$ by a subnet, we may assume that $u_n \in U_n$ for all $n$.   Now  Corollary~\ref{cor:conjugate_contraction} yields elements $t_n \in (U_n)_+$ such that $t_n\con(g)t\inv_n = \con(gu_n) = \con(g_n)$.  We have $t_n \to 1$ as $n \to \infty$ since $(U_n)_+$ is contained in $U_n$.  Similarly,  Corollary~\ref{cor:nub_conjugate} yields elements $r_n \in G$ such that $r_n\nub(g)r\inv_n = \nub(gu_n) = \nub(g_n)$ and such that $r_n \to 1$ as $n \to \infty$.
\end{proof}

\section{Normal closures}

We now establish a connection between the abstract normal closure of an element $f$ and the topological closure of its contraction group.

\begin{prop}\label{prop:ContractionGroup:NC}
Let $G$ be a \tdlc group and let $A$ be an abstract subgroup of $G$.  Let $g \in A$ and suppose that $K \le \N_G(A)$, where $K$ is a closed subgroup of $\overline{\con(g)}$ such that $gKg\inv = K$.  Then $K \le A$.  In particular, any (abstract) normal subgroup of $G$ containing $g$ also contains $\overline{\con(g)}$.
\end{prop}

\begin{proof}
The cyclic group generated by $g$ is discrete or relatively compact. In the latter case, $\con(g)=1$. The desired assertion is thus empty in this case. We assume henceforth that the cyclic group generated by $g$ is discrete. This implies that the subgroup  $H= \langle K \cup \{g \} \rangle$ is closed in $G$ and splits as
$$H \cong K \rtimes \langle g \rangle.$$ 
Let $N$ be the (abstract) normal closure of $g$ in $H$.  From now on we may assume $H = G$, so that $K = \overline{\con(g)}$. 

By \cite[Corollary~3.30]{BaumgartnerWillis}, we have $\overline{\con(g)} = \nub(g)\con(g)$. Moreover, it follows from \cite[Prop.~7.1]{WilNub} that $N$ contains $\nub(g)$. 

We now invoke the Tree Representation Theorem from \cite[Theorem~4.2]{BaumgartnerWillis}. This provides a locally finite tree $T$ and a continuous homomorphism 
$\rho \colon H \to \Aut(T)$ enjoying the following properties: 
\begin{itemize}
\item $\rho(f)$ acts as a hyperbolic isometry with attracting fixed point $\xi_+ \in \bd T$ and repelling fixed point $\xi_- \in \bd T$;  

\item $\rho(H)$ fixes $\xi_-$ and is transitive on $\bd T \setminus  \{\xi_-\}$;

\item the stabiliser $H_{ \xi_+}$ coincides with $\nub(f) \rtimes \langle f \rangle$. 
\end{itemize}
 
Any element $h \in H$ acting as a hyperbolic isometry fixes exactly two ends of $T$; one of them must thus be $\xi_-$. Since $H$ is   transitive on $\bd T \setminus  \{\xi_-\}$, it follows that some conjugate of $h$ is contained in $ H_{ \xi_+} $. We have seen that  $N$ contains $\nub(g) \rtimes \langle g \rangle = H_{ \xi_+} $. We infer that $N$ contains all elements of $H$ acting as hyperbolic isometries on $T$. 

Let now $\eta \in \bd T \setminus  \{\xi_-\}$. There is some $h \in H $ such that $\rho(h).\xi_+ = \eta$. Using again the fact that $\rho(H)$ fixes $\xi_-$, we remark that if $h$ is not a hyperbolic isometry, then $hg$ is a hyperbolic isometry. Moreover we have $\rho(hg).\xi_+ = \eta$. Recalling that $N$ contains all hyperbolic isometries of $H$, we infer that $N$ is transitive on $\bd T  \setminus  \{\xi_-\}$.  Therefore $N = H$ since $N$ also contains $H_{\xi_+}$.

This proves that the normal closure of $g$ in $H$ contains $K$.  Since by assumption $H$ normalises $A$, we conclude that $K \le A$.\end{proof}

By combining Corollary~\ref{cor:conjugate_contraction} and Proposition~\ref{prop:ContractionGroup:NC}, we arrive at a proof of Theorem~\ref{thm:condensenorm}.

\begin{proof}[Proof of Theorem~\ref{thm:condensenorm}]
Let $g \in G$.  Since $D$ is dense in $G$, there is a net $(g_n)$ in $D$ that converges to $g$.  By hypothesis, we have $\overline{\con(g_n)} \leq G^\dagger \leq \N_G(D)$ for all $n$. Applying  Proposition~\ref{prop:ContractionGroup:NC} with $K = \overline{\con(g_n)}$ and $A=D$, we infer that $\overline{\con(g_n)}$ is contained in $D$ for all $n$.  Moreover, by Corollary~\ref{cor:conjugate_contraction}, for some $n$ there is $t \in \con(g\inv)$ such that $t\con(g)t\inv = \con(g_n)$, and consequently $t\overline{\con(g)}t\inv = \overline{\con(g_n)}$.  Since $t \in \con(g\inv) \le \N_G(D)$ by assumption, it follows that $\overline{\con(g)} \le D$.  As $g$ was arbitrary, we conclude that $G^\dagger \le D$.
\end{proof}

We point out the following refinement of Corollary~\ref{cor:Titscore:abssimple}.

\begin{cor}\label{cor:abssim}Let $G$ be a \tdlc group.  Then the following are equivalent:
\begin{enumerate}[(i)]
\item $G$ is topologically simple, and $\con(g) \not=1$ for some $g \in G$;
\item $G^\dagger$ is abstractly simple and dense in $G$, and the centre of $G$ is trivial.
\end{enumerate}
\end{cor}

\begin{proof}Suppose that $G$ is topologically simple and that $\con(g) \not=1$ for some $g \in G$.  Evidently $G$ is non-abelian, so $\Z(G)$ is a proper closed normal subgroup of $G$, hence $\Z(G)=1$.  Moreover, $G^\dagger$ is a non-trivial normal subgroup of $G$, and in fact $G^\dagger$ must be dense in $G$ by topological simplicity.  Let $N$ be a non-trivial normal subgroup of $G^\dagger$.  Then $\N_G(\overline{N}) \ge \overline{(G^\dagger)}=G$, since the normaliser of any closed subgroup of $G$ is closed; thus $N$ is a dense subnormal subgroup of $G$, and thus $N = G^\dagger$ by Theorem~\ref{thm:condensenorm}.  Hence $G^\dagger$ is abstractly simple.

Conversely, suppose that $G^\dagger$ is abstractly simple and dense in $G$, and that $\Z(G)=1$.  Then in particular $G^\dagger$ is non-trivial, so $\con(g) \not=1$ for some $g \in G$.  Moreover, given a closed normal subgroup $N$ of $G$, then either $N \cap G^\dagger = 1$ or $N \ge G^\dagger$, and in the latter case $N=G$.  But if $N \cap G^\dagger = 1$, then $G^\dagger \le \CC_G(N)$, so in fact $N$ is central in $G$, using the fact that the centraliser of a subgroup is closed, and hence $N$ is trivial by the assumption that $\Z(G)=1$.  Hence $G$ is topologically simple.\end{proof}


\begin{thebibliography}{99}

\addcontentsline{toc}{section}{\protect\numberline{}Bibliography}

\bibitem{BaumgartnerWillis}
U. Baumgartner, G. A. Willis, Contraction groups and scales of automorphisms of totally disconnected locally compact groups,  
Israel J. Math. 142 (2004), 221--248. 


\bibitem{Jaworski}
W. Jaworski, 
On contraction groups of automorphisms of totally disconnected locally compact groups. 
Israel J. Math. 172 (2009), 1--8. 

\bibitem{Marquis}
T. Marquis, Abstract simplicity of locally compact Kac--Moody groups. Preprint (2013),  \url{http://arxiv.org/abs/1301.3681}.

\bibitem{Prasad}
G. Prasad, 
Strong approximation for semi-simple groups over function fields.
Ann. of Math. (2) 105 (1977), no. 3, 553--572. 

\bibitem{Tits64}
J. Tits, Algebraic and abstract simple groups.
Ann. of Math. (2) 80 (1964), 313--329. 

\bibitem{Tits89}
J. Tits, Groupes associ\'es aux alg\`ebres de Kac--Moody, Ast\'erisque (1989), no. 177--178, Exp. No. 700, 7--31, S\'eminaire Bourbaki, Vol. 1988/89.

\bibitem{Willis94}
G. A. Willis, The structure of totally disconnected, locally compact groups.
Math. Ann. 300 (1994), no. 2, 341--363. 

\bibitem{Willis01}
G. A. Willis,  Further properties of the scale function on a totally disconnected group.  
J. Algebra 237 (2001), no. 1, 142--164. 

\bibitem{Willis}
G. A. Willis, Compact open subgroups in simple totally disconnected groups, J. Algebra 312 (2007), no. 1, 405--417.

\bibitem{WilNub}
G. A. Willis, The nub of an automorphism of a totally disconnected, locally compact group, Ergodic Theory Dynam. Systems (to appear), preprint available at \url{http://arxiv.org/1112.4239}.

\end{thebibliography}
\end{document}